\documentclass[twoside,11pt]{article}

\usepackage{blindtext}
\usepackage{comment}

\usepackage[preprint]{jmlr2e}

\usepackage{optidef}
\usepackage{xcolor}

\usepackage{mathrsfs, amsmath, amssymb, mathtools, amsfonts}
\DeclareMathSymbol{\shortminus}{\mathbin}{AMSa}{"39}
\newcommand{\T}{^{\!\top}\!}

\DeclareMathAlphabet{\mymathbb}{U}{bbold}{m}{n}

\newcommand{\eye}{\mathbb{I}}

\newcommand{\inv}{^{\shortminus 1}}

\newcommand{\U}{\mathcal{U}}

\newcommand{\Z}{\mathcal{Z}}

\newcommand{\thetaSet}{\Theta}

\newcommand{\llsobjective}{\psi}

\newcommand{\ntheta}{n_{\theta}}

\newcommand{\NsmallerN}{\mathbb{N}_{< N}}

\newcommand{\problem}{\mathcal{P}}

\newcommand{\expected}[2]{\mathbb{E}_{#1}\left[#2\right]}

\DeclareMathOperator{\Prop}{\mathbb{P}}

\DeclareMathAlphabet{\mymathbb}{U}{bbold}{m}{n}

\newcommand{\bigO}{\mathcal{O}}

\newcommand{\R}{\mathbb{R}}

\newcommand{\thetatrue}{\theta^\star}

\newcommand{\info}{\mathcal{I}}

\newcommand{\uopttrue}{u^\star}

\newcommand{\objective}{\varphi}

\DeclareMathOperator{\trace}{tr}

\usepackage{lastpage}
\jmlrheading{??}{2025}{1-\pageref{LastPage}}{3/25; Revised ??/25}{??/25}{21-0000}{Katrin Baumgärtner, L\'eo Simpson, and Moritz Diehl}

\ShortHeadings{}{Baumgärtner, Simpson, Diehl}
\firstpageno{1}

\begin{document}

\title{Gray-Box Optimization using Optimism in the Face of Uncertainty}

\author{\name Katrin Baumgärtner \email katrin.baumgaertner@imtek.uni-freiburg.de\\
       \addr Department of Microsystems Engineering \\
       University of Freiburg\\
       Germany
       \AND
       \name L\'eo Simpson \email leo.simpson@imtek.uni-freiburg.de\\
       \addr Department of Microsystems Engineering and Department of Mathematics\\
       University of Freiburg\\
       Germany
       \AND
       \name Moritz Diehl \email moritz.diehl@imtek.uni-freiburg.de\\
       \addr Department of Microsystems Engineering and Department of Mathematics\\
       University of Freiburg\\
       Germany}

\editor{My editor}

\maketitle

\begin{abstract}%
This paper considers sequential gray-box optimization where the objective function is given as the composition of a loss function and a parametric model.
Crucially, the parameters of the model are unknown and need to be iteratively estimated from noisy observations of the model outputs.
This problem setup generalizes the parametric black-box optimization problem known as (contextual) stochastic linear bandit.
To address the sequential gray-box optimization problem, we propose a structure-exploiting method that leverages known problem structure given in terms of the loss function and an a priori set of admissible parameters.
The method is based on the principle of \textit{optimism in the face of uncertainty} and trades off exploration and exploitation by minimizing a lower confidence bound on the true objective function.
We provide a detailed regret analysis of the novel method, improving on state-of-the-art results for the special case of linear stochastic bandits due to the use of a recently published bound for the parameter confidence sets arising in multi-output linear least-squares estimation.
Numerical examples illustrate the superior performance of structure-exploiting methods compared to structure-agnostic approaches.
\end{abstract}

\begin{keywords}
gray-box optimization, optimism in the face of uncertainty, linear stochastic bandit, contextual bandit, regret bound
\end{keywords}

\section{Introduction}

In many real-world decision-making problems, optimization must be carried out under limited information and uncertainty.
A common modeling paradigm in this context is black-box optimization where the objective function is treated as an unknown mapping from actions to the objective value \citep{Mockus1978, Brochu2010}.
While this abstraction enables broad applicability \citep{Falkner2018, Berkenkamp2016, Berkenkamp2023, Kudva2024}, it often ignores valuable prior knowledge about how the objective is generated.
As a result, black-box methods may require substantial data to achieve satisfactory performance, which might limit their performance in settings where decisions need to be made online.

In contrast, many applications -- for instance in operations research, engineering, control, or robotics -- naturally exhibit partial structure that can be explicitly modeled:
The objective function often arises as the composition of a known loss function with an underlying parametric model.
Although the parameters governing this model are unknown, the functional relationship between unknown parameters, actions, and observed outputs is often well understood.
This setting motivates the study of gray-box optimization, which lies between the extremes of fully specified (white-box) and completely unknown (black-box) models.
By incorporating known structural components while accounting for parametric uncertainty, gray-box approaches offer the potential for significantly improved data efficiency and decision quality \citep{Astudillo2021}.

\begin{figure}
\centering
\includegraphics[width=0.9\linewidth]{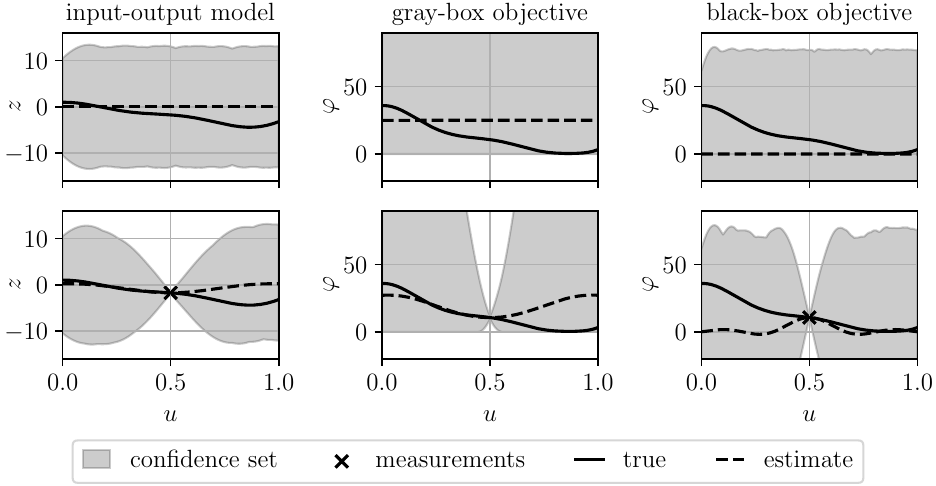}
\vspace*{10pt}
\caption{Black-box vs. gray-box modelling:
We consider a scalar input-output model $z = a_{0}
+\sum_{i=1}^{K}
\left(
a_{i}\cos(i u)
+b_{i}\sin(i u)
\right)$, where $\theta = (a_0, \ldots, a_K, b_1, \ldots, b_K)$, which corresponds to a Fourier series of order $K = 10$, and a loss $l(u, z) = 0.01 u^2 + z^2$.
The a priori estimate of the model and the corresponding confidence set (which is discussed in more detail in Section~\ref{sec:lls}) are shown in the top left plot.
The bottom left plot shows model and confidence sets after evaluation at $u = 0.5$ and observing a corresponding measurement. 
The middle column shows the associated values of the objective $\varphi$ within a gray-box approach, i.e., the objective is given as the composition of the input-output model and the given loss.
While the objective is a priori highly uncertain, it is lower bounded by zero as induced by the structure of the loss.
On the other hand, when considering the same problem within a black-box framework, where we directly model the objective as a parametric function that is linear in the parameters, positiveness of the objective is not preserved as illustrated by the two plots on the right. 
}
\label{fig:fourier}
\end{figure}

\newpage

Motivated by these considerations, this paper investigates a sequence of gray-box optimization problems of the form
\begin{align}\label{eq:model-intro}
u_n^\star \in \arg \min_{u \in \U_n} ~ l_n(u, A_n(u)\thetatrue),
\end{align}
for $n \in \NsmallerN$.
In~\eqref{eq:model-intro}, the decision variable is the action $u \in \U_n$ and the objective function $\objective_n(u, \thetatrue) := l_n(u, A_n(u)\thetatrue)$ at time step $n$ is given by the composition of a known loss function $l_n: \U_n \times \R^{n_z} \rightarrow \R$ and a parametric model $z = A_n(u)\thetatrue$ defined by a known matrix-valued function $A_n: \U_n \rightarrow \R^{n_z \times n_\theta}$ and the true parameter value $\thetatrue \in \thetaSet \subset \R^{\ntheta}$.

Since $\thetatrue$ is unknown, we cannot directly solve \eqref{eq:model-intro} to obtain the optimal action $\uopttrue_n$.
The sequential gray-box optimization task thus requires leveraging the dual control effect \citep{Feldbaum1960}:
After deciding on an action $u_n$, we observe a noisy measurement $y_n = A_n(u_n)\thetatrue + v_n$ of the multivariate model outputs where $v_n$ denotes zero-mean noise.
Each action $u_n$ therefore serves two purposes.
On the one hand, it seeks to optimize the immediate costs (\textit{exploitation}).
On the other hand, it probes the system in regions where improved knowledge may lead to improved future decisions (\textit{exploration}).
In contrast to the black-box setting, the gray-box setting allows the decision-maker to infer information about the unknown parameters from (potentially multivariate) measurement of the model outputs, as illustrated in Figure~\ref{fig:fourier}.

To address the sequential gray-box problem, we propose a novel structure-exploiting method that leverages both the known loss function $l_n$ and prior information about the admissible parameter space given by $\thetaSet$.
The approach is based on the principle of \textit{optimism in the face of uncertainty} \citep{Bubeck2012}, where each action is chosen to minimize a lower confidence bound (LCB) on the true objective function.
This principle is motivated by the following intuition:
If the uncertainty associated with objective value -- where uncertainty is in the following described by a confidence set -- is large at the point of evaluation, the gathered observation will improve the knowledge about the unknown parameter (\textit{exploration}).
If the lower bound is tight at the minimum, the action will necessarily incur low costs (\textit{exploitation}).
Therefore, this strategy naturally balances the trade-off between exploration and exploitation.
In contrast to explicit dual control approaches \citep{FIlatov2000}, which enforce this trade-off through an explicit regularization term and therefore often induce \textit{undirected} exploratory behavior, the proposed optimistic approach captures the same principle implicitly.
In the time-invariant setting or if subsequent problem instants are similar, this leads to a more targeted exploration strategy.
Indeed, exploration emerges only when uncertainty can plausibly conceal a better solution, and is therefore restricted to directions that have the potential to reduce the objective.

We analyze the performance of the proposed method in terms of regret, building upon established results from the linear bandit literature \citep{Dani2008, Abbasi-Yadkori2011, Rusmevichientong2010, Lattimore2020}.
By integrating recent advances in the construction of data-dependent parameter confidence sets for least-squares estimators \citep{Simpson2026},
our regret analysis provides, as a special case, improved regret bounds for the stochastic linear bandit and contextual bandit problem compared to the state of the art.

Finally, we complement the theoretical findings with numerical examples that demonstrate the practical advantages of exploiting problem structure.
The results highlight that structure-exploiting methods can substantially outperform structure-agnostic black-box approaches, emphasizing the importance of moving beyond purely black-box formulations when additional knowledge is available.

\subsection{Related Work}

Similar to this paper, \cite{Chowdhury2021} and \cite{Wang2026} propose to address the sequential gray-box optimization problem with methods based on the principle of optimism in the face of uncertainty.
In contrast to the parametric models used in this work, they consider nonparametric kernel regression to identify a multivariate operator that lies in a reproducing kernel Hilbert space.
Regarding the model class, their approaches thus generalize our problem setting.
The focus on parametric models, however, allows us to take into account an a priori set of admissible parameters, which is not possible within the more general nonparametric setting.
In addition, the approach proposed by \cite{Chowdhury2021} does not exploit the full structure provided by the known loss function within the subproblems solved in each iteration, but instead only leverages knowledge of its Lipschitz constant.
\cite{Wang2026} focus on linear loss functions, while our approach allows for general Lipschitz-continuous loss functions.

Similarly, gray-box optimization using Gaussian processes as the model class is investigated by \cite{Astudillo2019} and \cite{Astudillo2021} under the term Bayesian optimization with composite functions.
They, however, focus on the expected improvement acquisition function, which is optimized using first-order optimization techniques based on Monte Carlo (MC) estimates of the gradient.
\cite{Paulson2022} consider a similar gray-box Bayesian optimization setting but in addition allow for (known) constraints on actions and model outputs.
They also consider a variant of the expected improvement acquisition function with MC gradient estimates.
In addition, the software framework presented in \cite{Balandat2020} allows for gray-box Bayesian optimization and provides efficient MC gradient estimates for various acquisition functions such as expected improvement and knowledge gradient.
The approaches mentioned above rely on first-order optimization based on MC gradient estimates of the acquisition function.
In contrast, our approach does not require MC sampling and can leverage second-order optimization methods.

While parametric models are not commonly considered in the context of gray-box Bayesian optimization, they are widely used within sequential black-box optimization.
In particular, the linear bandit and the contextual linear bandit problems arise as special cases of our sequential gray-box optimization setup:
If the loss function is of the form $l_n(u, z) = z$, and the scalar parametric model is linear in the actions, $y_n = u_n\T \,\thetatrue + v_n$, we recover the stochastic linear bandit problem introduced by \cite{Abe1999} originally considering a finite number of actions.
If the loss is $l_n(u, z) = z$, but the actions enter the scalar parametric model nonlinearly, $y_n = a_n(u_n)\T \thetatrue + v_n$, the problem reduces to the contextual bandit problem \citep{Li2010}.
\cite{Auer2002} first proposed an algorithm for linear bandits based on the principle of \textit{optimism in the face of uncertainty}.
Further significant contributions to the analysis of LCB-type algorithms for linear bandits, which our regret analysis builds on, have been made by \cite{Dani2008}, \cite{Rusmevichientong2010}, and \cite{Abbasi-Yadkori2011}.

At first glance, the investigated gray-box setting might resemble generalized linear bandits \citep{Nelder1972} for which a confidence bound algorithm has been introduced by \cite{Filippi2010}.
With generalized linear bandits, however, the objective function is given as $l(a_n(u)\T \theta)$, a composition of a scalar monotonic \textit{link} function $l(\cdot)$ with a single-output parametric model.

For all of the above variations of the bandit problem, the quantity being measured is always the scalar objective value.
This is a key difference to our setup, as we allow the use of multivariate measurements of the model outputs.

From a more applied perspective, the investigated gray-box setting is typical for control and robotics applications where the loss as a function of the control actions and the system outputs is known -- and might be even chosen or tuned by the control engineer -- while an accurate system model relating the control actions to the system outputs is not available.
Control applications covered by this problem setup comprise steady-state optimization, as analyzed for example within the field of modifier adaptation \citep{Marchetti2009a}, and iterative learning control where typically both actions and outputs are trajectories arising from the discretization of a continuous-time state-space model \citep{Bristow2006, Volckaert2010, Schoellig2012a}.

\subsection{Contribution}

The main contribution of this paper is a structure-exploiting method for sequential gray-box optimization with parametric models based on the principle of \textit{optimism in the face of uncertainty}.
Unlike existing approaches, the proposed method leverages both the known structure of the loss function and prior knowledge in the form of an admissible parameter set.
We further provide a regret analysis of the algorithm, which, as a special case, improves upon the state-of-the-art bounds for contextual linear stochastic bandits.

\subsection{Outline}

The remainder of the paper is structured as follows.
Section~\ref{sec:setup} introduces the problem setup and defines the regret, which will be used as a performance criterion for the proposed algorithm.
Section~\ref{sec:lls} briefly recalls linear least-squares estimation and the construction of data-dependent confidence sets for the estimated parameter.
These confidence sets are used in Section~\ref{sec:problem} to devise the novel structure-exploiting approach for sequential gray-box optimization.
In Section~\ref{sec:regret}, we provide a detailed regret analysis for the proposed method.
Section~\ref{sec:linear-bandit} discusses the theoretical findings for the linear bandit problem, which arises as a special case of the sequential gray-box optimization setting.
In Section~\ref{sec:experiments}, we contrast the proposed structure-exploiting gray-box method with the structure-agnostic black-box approach on illustrative examples.
Section~\ref{sec:conclusion} concludes the paper with a discussion of the results and an outlook on future research directions.

\subsection{Notation}

We use the shorthand notation $(x, y) := [x\T ~y\T]\T \in \R^{n+m}$ to denote vertical composition of column vectors $x\in \R^n$, $y \in \R^m$.
The weighted norm $\Vert x \Vert_M$ is given by $\Vert x \Vert_M = \sqrt{x\T M x}$ for any positive semidefinite matrix $M$.
In addition, we use $\Vert M \Vert_\mathrm{F}$ to denote the Frobenius norm of matrix $M$.
For square matrices $M$, $\rho(M)$ denotes the spectral radius.
We furthermore use the notation $\NsmallerN = \left\lbrace 0, 1, \ldots, N-1\right\rbrace$.

\section{Problem statement}
\label{sec:setup}

In the following, we introduce the potentially time-varying sequential gray-box optimization problem and define the performance criterion that will be used to evaluate the proposed approach.

For $n \in \NsmallerN$, we consider a time-varying input-output model
defined by some known matrix-valued function $A_n: \U_n \rightarrow \R^{n_z \times \ntheta}$ and an unknown parameter $\thetatrue \in  \thetaSet \subset \R^{\ntheta}$:
\begin{align}
y_n = A_n(u_n)\thetatrue + v_n,
\end{align}
where $v_n$ denotes zero-mean measurement noise.
We call $u_n$ the input or action, $z_n = A_n(u_n)\thetatrue$ the model output, and $y_n = A_n(u_n)\thetatrue + v_n$ its noisy observations.
While we do not know the true parameter $\thetatrue$, we assume that a set $\thetaSet \subset \R^{\ntheta}$ is given that contains $\thetatrue$ and encodes prior knowledge about admissible parameter values.
For instance, in applications where $\theta$ represents concentrations, such as the example discussed in Section~\ref{sec:example-steel}, a natural choice for the set of admissible parameters is $\thetaSet = [0, 1]^{\ntheta}$.

After choosing an admissible action $u_n \in \U_n$ at time step $n$, we receive the noisy measurement $y_n$, which can be used to estimate the unknown parameter.
In addition, the chosen action $u_n$ incurs a cost $\objective_n(u_n, \thetatrue)$ given by
\begin{align}
\objective_n(u_n, \thetatrue) := l_n(u_n, A_n(u_n)\thetatrue),
\end{align}
where  $l_n: \U_n \times \R^{n_z} \rightarrow \R$ is a known loss function.
In contrast to the generalized linear bandit setting, we do not observe a noisy measurement of $\objective_n(u_n, \thetatrue)$ after choosing action $u_n$.

Let $\info_{n} = (u_0, y_0, \ldots, u_{n-1}, y_{n-1})$ denote the available data at time step $n$.
We aim to find a time-varying policy $u_n = \pi_n(\info_{n})$ that minimizes the total costs over a horizon $N$.
As $\thetatrue$ is unknown, the policy needs to trade off exploration and exploitation in order to ensure, on the one hand, informative measurements to estimate $\thetatrue$, while, on the other hand, keeping the incurred loss small.

To evaluate the performance of such a policy $u_n = \pi_n(\info_{n})$ over a horizon $N$, we consider the cumulative regret $R_N$ defined as
\begin{align}
R_N = \sum_{n=0}^{N-1} \objective_n(u_n, \thetatrue) - \objective_n(\uopttrue_n, \thetatrue),
\end{align}
where $\uopttrue_n$ is an optimal action at time step $n$ given as a solution to
\begin{mini}|s|
	{\scriptstyle{u \in \U_n}}
	{l_n(u, A_n(u)\thetatrue).}
	{}
	{\uopttrue_n \in \arg}
\end{mini}
We call $r_n = \objective_n(u_n, \thetatrue) - \objective_n(\uopttrue_n, \thetatrue)$ the instantaneous regret.

\begin{remark}[Frequentist vs. Bayesian approach]
In this paper, we take a frequentist approach meaning that the unknown parameter is considered deterministic and stochasticity enters the problem only via the measurement noise.
In contrast, a Bayesian approach would model the unknown parameter as a random variable.
The uncertainty is in this case quantified by the posterior distribution of the parameter given the data.
Even if the probability distributions of prior and noise are exactly known, deriving credible sets based on the posterior distribution is often intractable, with only rare exceptions such as the case of Gaussian prior and noise.
In contrast, our assumption on the noise, cf. Assumption~\ref{ass:subgaussian}, covers a variety of noise distribution and in particular does not require exact knowledge of the distribution.
In the Bayesian framework, the performance criterion is typically the expected cumulative costs and the optimal policy is characterized by the Bellman equation \citep{Bellman1957}.
Numerical computation of the optimal policy for general loss function is however intractable even if prior and noise distributions are Gaussian.
\end{remark}

\section{Linear least-squares estimation and data-dependent confidence sets}
\label{sec:lls}

Before introducing the novel structure-exploiting method, we briefly recall constrained linear least-squares estimation, which we will employ in the following to estimate the unknown parameter from the data $\info_n$.
In particular, we discuss how data-dependent confidence sets for the parameter can be constructed.

For $n \in \NsmallerN$, we define the constrained linear least-squares estimate $\mu_n(\info_n)$ of the parameter $\thetatrue$ based on the data $\info_n$ as
\begin{mini}|s|
	{\scriptstyle{\theta \in \thetaSet}}
	{\frac{1}{2}\left\Vert \theta - \mu_0 \right\Vert^2_{\Lambda_0} + \sum_{i=0}^{n-1} \frac{1}{2}\left\Vert y_i - A_i(u_i)\theta \right\Vert^2_{V},}
	{\label{eq:lls}}
	{\mu_n(\info_n) = \arg}
\end{mini}
for positive definite weighting matrices $\Lambda_0 \succ 0$, $V \succ 0$, where $\mu_0$ and $\Lambda_0$ are given regularization parameters.
Let $\llsobjective_n(\theta; \info_n)$ denote the objective function of \eqref{eq:lls}.
We denote its Hessian by $\Lambda_n(\info_n) := \nabla_{\theta}^2 \llsobjective_n(\theta; \info_n)$, it is explicitly given by
\begin{align} \label{eq:lls-hessian}
\Lambda_n(\info_n) = \Lambda_0 + \sum_{i=0}^{n-1} A_i(u_i)\T V A_i(u_i).
\end{align}

\begin{assumption}[Conditional subgaussian noise.] \label{ass:subgaussian}
For all $n \in \NsmallerN$, the measurement noise $v_n$ conditioned on $(\info_n, u_n)$ is zero-mean subgaussian with proxy variance $c_v^2 V\inv$ where $c_v > 0$ is a known constant.
\end{assumption}

Note that Assumption~\ref{ass:subgaussian} is satisfied for $v_n \sim \mathcal{N}(0, \Sigma^v)$ if $\rho(\Sigma^v V) \leq c_v^2$.
Furthermore, the assumption covers zero-mean bounded noise, i.e., the case where $\Vert v_n \Vert_V \leq c_v$ almost surely and $\expected{}{v_n|\info_{n}, u_n} = 0$.

\begin{assumption}[Bounded and convex set of admissible parameters.] \label{ass:prior}
Let $\Lambda_0$ and $\mu_0$ be the regularization parameters in the constrained linear least-squares estimation problem \eqref{eq:lls}.
The admissible set of parameters $\thetaSet$ is convex and contains both $\thetatrue$ and $\mu_0$.
Furthermore, there is a constant $c_\theta > 0$ such that $\Vert \theta - \mu_0 \Vert_{\Lambda_0} \leq c_\theta$ for all $\theta \in \thetaSet$.
\end{assumption}

Under these assumptions, recent work on linear least-squares estimation \citep{Simpson2026}, improving on results in \cite{Abbasi-Yadkori2011} and most importantly generalizing these results to multi-output systems, implies the following uniform bound for constructing confidence sets based on the data $\info_n$.

\begin{theorem}[Data-dependent confidence bound.] \label{thm:sim-conf-bound-data-dependent-constrained}
Suppose Assumptions~\ref{ass:subgaussian} and \ref{ass:prior} hold.
For any confidence level $\delta \in (0, 1)$, we  have
\begin{align}
\mathbb{P}\left[\Vert \thetatrue - \mu_n\Vert_{\Lambda_n} \leq \gamma_n(\delta), \forall n \in \NsmallerN \right] \geq 1 - \delta,
\end{align}
for
\begin{align} \label{eq:data-dependent-gamma-constrained}
\gamma_n(\delta) := \sqrt{c_\theta^2 + c_v^2 \log\left(\det\left(\Lambda_0\inv \Lambda_n\right)\right) + 2 c_v^2 \log\left(\tfrac{1}{\delta}\right)}.
\end{align}
\end{theorem}

\begin{proof}
First, note that the objective in \eqref{eq:lls} can be expressed as
\begin{align}
\llsobjective_n(\theta) = \frac{1}{2} \Vert \theta - \thetatrue \Vert^2_{\Lambda_n} + \left(\Lambda_0(\thetatrue - \mu_0) - s_n\right)\T \,(\theta - \thetatrue) + \text{ constant},
\end{align}
where $s_n =  \sum_{i=0}^{n-1} A_i(u_i)\T V v_i$.
	As $\thetaSet$ is convex and $\llsobjective_n$ is a convex quadratic function, the constrained linear least-squares problem \eqref{eq:lls} is a convex optimization problem.
	Therefore, by optimality of $\mu_n$, we conclude
	\begin{align}
		\nabla \llsobjective_n(\mu_n)\T (\theta - \mu_n) \geq 0 \text{ for all } \theta \in \thetaSet.
	\end{align}
	In particular, since $\theta^\star \in \thetaSet$, we have
	\begin{align}
		-\bigg(
			\Lambda_n (\mu_n - \thetatrue) + \left(\Lambda_0(\thetatrue - \mu_0) - s_n\right)
		\bigg)\T  (\mu_n - \thetatrue) \geq 0.
	\end{align}
	After rearranging, we obtain
	\begin{align}
	\Vert \mu_n - \thetatrue \Vert^2_{\Lambda_n}
	&\leq -\left(\Lambda_0(\thetatrue - \mu_0) - s_n\right)\T (\mu_n - \thetatrue) \\
	&=
	-\bigg(\Lambda_n^{-\frac{1}{2}} \left(\Lambda_0(\thetatrue - \mu_0) - s_n\right) \bigg)\T \,\Lambda_n^{\frac{1}{2}}(\mu_n - \thetatrue)
	.
	\end{align}
Applying the Cauchy-Schwarz inequality yields
\begin{align}
\Vert \mu_n - \thetatrue \Vert^2_{\Lambda_n} \leq \left\Vert \Lambda_0(\mu_0 - \thetatrue) + s_n \right\Vert_{\Lambda_n\inv} \left\Vert \mu_n - \thetatrue \right\Vert_{\Lambda_n}.
\end{align}
Dividing by $\Vert \mu_n - \thetatrue \Vert_{\Lambda_n}$, we arrive at
$ \Vert \mu_n - \thetatrue \Vert_{\Lambda_n} \leq \left\Vert \Lambda_0(\mu_0 - \thetatrue) + s_n \right\Vert_{\Lambda_n\inv}$.
Applying Theorem~1 from \cite{Simpson2026} yields the desired result.
\end{proof}

Note that we can also choose $\gamma_0 = c_\theta$, which is strictly smaller than \eqref{eq:data-dependent-gamma-constrained} for $n=0$, and obtain the same bound due to Assumption~\ref{ass:prior}.
Furthermore, we point out that Theorem~\ref{thm:sim-conf-bound-data-dependent-constrained} coincides with the bound in Theorem 2 in \cite{Simpson2026} for unconstrained linear-least squares estimation.
For a discussion of the improvement of the above bound in comparison to the results in \cite{Abbasi-Yadkori2011}, we refer to \cite{Simpson2026}.

The method proposed in the next section can be applied with either the unconstrained or constrained linear least-squares estimator.
Depending on the application, it may be essential that the estimate $\mu_n$ satisfies $\mu_n \in \thetaSet$.
Again considering concentrations as an example, a negative estimate might result in unreasonable or even ill-defined model outputs.
If the condition $\mu_n \in \thetaSet$ is not critical, the constraint might be omitted for computational efficiency, since the unconstrained linear least-squares problem admits a closed-form solution.

\begin{remark}
The result in Theorem~\ref{thm:sim-conf-bound-data-dependent-constrained} can be extended to nonconvex admissible sets $\thetaSet$.
In this more general case, we, however, obtain $\tilde\gamma_n(\delta) = 2 \gamma_n(\delta)$ with $\gamma_n(\delta)$ as in \eqref{eq:data-dependent-gamma-constrained}, which follows from similar arguments as in the proof of Theorem~\ref{thm:sim-conf-bound-data-dependent-constrained}:

Due to optimality of $\mu_n \in \thetaSet$, we have
\begin{align}
\llsobjective_n(\mu_n) - \llsobjective_n(\thetatrue)
&=
\frac{1}{2} \Vert \mu_n - \thetatrue \Vert^2_{\Lambda_n} + \left(\Lambda_0(\thetatrue - \mu_0) - s_n\right)\T \,(\mu_n - \thetatrue)  \leq 0,
\end{align}
Rearranging and multiplying by two yields
\begin{align}
\Vert \mu_n - \thetatrue \Vert^2_{\Lambda_n} \leq -2\left(\Lambda_0(\thetatrue - \mu_0) - s_n\right)\T \,(\mu_n - \thetatrue).
\end{align}
Applying the Cauchy-Schwarz inequality and Theorem 1 from \cite{Simpson2026} yields the result.
\end{remark}

\section{A Novel Method for Sequential Gray-Box Optimization}
\label{sec:problem}

In the following, we introduce the novel structure-exploiting algorithm for gray-box optimization, which is based on the principle of \textit{optimism in the face of uncertainty}.
At each time step, the method minimizes a lower confidence bound (LCB) of the true, but unknown, objective, exploiting the known problem structure provided by the set of admissible parameter values $\thetaSet$ and by the loss function $l_n(u, z)$.
The proposed method builds on the key assumption that, given the available data $\info_n$, it is possible to construct a confidence set for the unknown parameter such that the true parameter value lies within this set with high probability.

At iteration $n \in \NsmallerN$ and given the data $\info_{n}$, as well as the confidence level $\delta$, the next evaluation point $u_{n}$ is selected as a solution to $\problem_n(\delta)$ defined as
\begin{mini!}|s|
	{\scriptstyle{u, \theta}}
	{l_n\left(u, A_n(u)\theta\right) }
	{\label{nlp:lcb-method}}
	{\problem_n(\delta): \quad }
	\addConstraint{\Vert \theta}{- \mu_{n}\Vert_{\Lambda_{n}} \leq \gamma_{n}(\delta) \label{con:confidence-set}}
	\addConstraint{u}{\in \U_n \label{con:u-set}}
	\addConstraint{\theta}{\in \thetaSet, \label{con:theta-set}}
\end{mini!}
with $\mu_{n}, \Lambda_{n}$ and $\gamma_{n}(\delta)$ as defined in \eqref{eq:lls}, \eqref{eq:lls-hessian} and \eqref{eq:data-dependent-gamma-constrained}, respectively.
Evaluation of the true model at $u_{n}$ yields the noisy observation $y_{n} = A_n(u_{n})\thetatrue + v_{n}$ and the procedure is repeated for problem $\problem_{n+1}(\delta)$ with updated confidence set based on the data $\info_{n+1} = (\info_{n}, u_{n}, y_{n})$.

We point out that problem \eqref{nlp:lcb-method} can be infeasible if unconstrained linear least-squares estimation is employed and the parameter confidence set is conflicting with the set of admissible parameters $\thetaSet$.
Under the given assumptions, however, the problems $\problem_n(\delta)$ are feasible for all $n \in \NsmallerN$ with high probability, as shown next.

\begin{proposition}[Feasibility]
Suppose Assumptions~\ref{ass:subgaussian} and \ref{ass:prior} hold.
For any confidence level $\delta \in (0, 1)$, the following statement holds with probability at least $1 - \delta$:
The problems $\problem_n(\delta)$ admit a feasible point for all $n \in \NsmallerN$.
\end{proposition}
\begin{proof}
If constrained linear least-squares is employed, feasibility is guaranteed as $\mu_n \in \thetaSet$ for all $n \in \NsmallerN$.
If unconstrained linear least-squares is employed, the statement is a direct consequence of Assumption~\ref{ass:prior} and Theorem~2 in \cite{Simpson2026}.
\end{proof}

The proposed approach can alternatively be defined in terms of the minimization of an acquisition function, in the following denoted by $Q_n(u; \delta)$, which is a lower bound on the true objective $\objective_n(u, \thetatrue)$ with high probability.
The proposed approach is thus \textit{optimistic} as it underestimates the true costs associated with the action $u$.

\begin{definition}[Acquisition function]
The next evaluation point $u_{n}$, given as a solution to $\problem_n(\delta)$, can equivalently be obtained as a minimizer of the acquisition function $Q_n(u; \delta)$,
\begin{mini*}|s|
	{\scriptstyle{u \in \U_n}}
	{Q_n(u; \delta),}
	{}
	{u_{n} \in \arg}
\end{mini*}
where $Q_n(u; \delta)$ is defined as
\begin{mini!}|s|
	{\scriptstyle{\theta \in \thetaSet}}
	{l_n\left(u, A_n(u)\theta\right) }
	{\label{nlp:lcb-method-acquisition-function}}
	{Q_n(u; \delta) :=}
	\addConstraint{\Vert \theta}{- \mu_{n}\Vert_{\Lambda_{n}} \leq \gamma_{n}(\delta).}
\end{mini!}
\end{definition}

\begin{lemma}[Lower confidence bound] \label{lem:lcb}
Suppose Assumptions~\ref{ass:subgaussian} and \ref{ass:prior} hold.
For a confidence level $\delta \in (0, 1)$, the acquisition function $Q_n(u; \delta)$ is a lower bound on $\objective_n(u, \thetatrue)$ on the domain $\U_n$ for all $n \in \NsmallerN$ with probability at least $1 - \delta$,
\begin{align*}
	\Prop\left[Q_n(u; \delta) \leq  \objective_n(u, \thetatrue), ~\forall u \in \U_n, \forall n \in \NsmallerN\right] \,\geq\, 1 - \delta.
\end{align*}
\end{lemma}

\begin{proof}
Consider $u \in \U_n$ and let $\tilde{\theta}_n$ be a solution to \eqref{nlp:lcb-method-acquisition-function} for the given action $u$.
If $\Vert \thetatrue- \mu_{n}\Vert_{\Lambda_{n}} \leq \gamma_{n}(\delta)$ and $\thetatrue \in \thetaSet$, then $\thetatrue$ is a feasible point of \eqref{nlp:lcb-method-acquisition-function}.
This implies
\begin{align*}
Q_n(u; \delta)
	= l_n(u, A_n(u)\tilde\theta_n)
	\leq l_n(u, A_n(u)\thetatrue)
	= \objective_n(u, \thetatrue).
\end{align*}
Theorem~\ref{thm:sim-conf-bound-data-dependent-constrained} together with Assumption~\ref{ass:prior} implies that $\Vert \thetatrue- \mu_{n}\Vert_{\Lambda_{n}} \leq \gamma_{n}(\delta)$ and $\thetatrue \in \thetaSet$ for all $n \in \NsmallerN$ holds with a probability of at least $1 - \delta$, which concludes the proof.
\end{proof}

An important property highlighting that the proposed method preserves known problem structures is the following:
\begin{proposition}
	If the objective function is bounded uniformly in $\theta \in \thetaSet$, then the acquisition function $Q_n(u; \delta)$ will inherit the same bounds:
	If
	\begin{align}
		\forall \theta \in \thetaSet, u \in \U_n: \quad \objective^{\min}_n(u) \le \objective_n(u, \theta) \le \objective^{\max}_n(u),
	\end{align}
	for some functions $\objective^{\min}_n(u)$ and $\objective^{\max}_n(u)$, then we also have:
	\begin{align}
		\forall u \in \U_n: \quad \objective^{\min}_n(u) \le Q_n(u; \delta) \le \objective^{\max}_n(u).
	\end{align}
\end{proposition}
\begin{proof}
This directly follows from $Q_n(u; \delta) = \objective_n(u, \tilde{\theta}(u; \delta) )$ for some $\tilde{\theta}(u; \delta) \in \thetaSet$.
\end{proof}

Boundedness of the acquisition function does in general not hold in the black-box setting, where the acquisition function might tend to $-\infty$ as $\delta \rightarrow 1$.
This will be illustrated in Section~\ref{sec:experiments} where we compare the proposed method to the structure-agnostic LCB approach.

\begin{remark}[A different acquisition function]
In~\cite{Chowdhury2021}, an explicit lower bound of $Q_n(u; \delta)$ is used as their acquisition function, which is derived from a Lipschitz assumption on the loss function.
This approach allows a similar regret analysis as the one provided here, but does not exploit the full structure provided by the known loss function and the parameter confidence set.
In particular, their acquisition function does not preserve boundedness.
\end{remark}

\begin{remark}
We point out that the confidence set defined by $\Vert \theta - \mu_{n}\Vert_{\Lambda_{n}} \leq \gamma_{n}(\delta)$ is not necessarily a subset of the previous confidence set.
One might implement a variant of the proposed algorithm where multiple confidence sets are considered at each time step, i.e., the constraint \eqref{con:confidence-set} could be replaced by
$\Vert \theta - \mu_{i}\Vert_{\Lambda_{i}} \leq \gamma_{i}(\delta) \; \forall i \in \{0, \ldots, n\}$ in order to further reduce the search space.
As the parameter confidence bound given in Theorem~\ref{thm:sim-conf-bound-data-dependent-constrained} holds uniformly over the time steps, the regret analysis provided in the next section directly covers this variant of the proposed method.
If all previous confidence sets are considered and the problem is time-invariant, we obtain monotone acquisition functions $Q_n(u; \delta) \leq Q_{n+1}(u; \delta)$ for all $u \in \U$.
\end{remark}

\section{Regret Analysis}
\label{sec:regret}

In the following, we derive a regret bound for the proposed structure-exploiting optimistic method for sequential gray-box optimization.
The proof is based on established results from linear bandit theory.
In particular, we adapt results from \cite{Dani2008, Abbasi-Yadkori2011, Rusmevichientong2010, Lattimore2020}.
Crucially, we leverage an improvement over the classic results by \cite{Abbasi-Yadkori2011}, recently published in \cite{Simpson2026}, in order to obtain data-dependent parameter confidence sets.
It is important to keep in mind that the analysis is based on the assumption that we are able to compute a global minimizer of the acquisition function in each iteration.
Satisfying this assumption remains a major challenge in practice.

In order to derive a data-dependent high-probability bound on the cumulative regret, we make the following assumptions.

\begin{assumption}[Lipschitz-continuous loss] \label{ass:lipschitz-l}
Let $V$ be the weighting matrix in \eqref{eq:lls} and $\Z_n= \left\lbrace z \in \R^{n_z} \mid z = A_n(u)\theta, u \in \U_n, \theta \in \thetaSet \right\rbrace$.
We assume the loss functions $l_n(u, z)$ are Lipschitz-continuous in $z$ with respect to the norm $\Vert \cdot \Vert_V$ with shared Lipschitz constant $L_z$, i.e., for all $n \in \NsmallerN$, we have
\begin{align}
| l_n(u, z_1) - l_n(u, z_2) | \leq L_z \Vert z_1 - z_2 \Vert_V \text{ for } u \in \U_n, z_1, z_2 \in \Z_n.
\end{align}
\end{assumption}

\begin{assumption}[Bounded regret] \label{ass:bounded-regret}
We assume the instantaneous regret is bounded by a constant $c_r > 0$,
\begin{align}
\objective_n(u_n, \thetatrue) - \objective_n(\uopttrue_n, \thetatrue) \leq c_r, \forall u_n \in \U_n, \forall n \in \NsmallerN.
\end{align}
\end{assumption}

We point out that Assumption~\ref{ass:bounded-regret} is not necessarily needed for our regret analysis, as discussed in more detail in Remarks~\ref{remark:bounded-regret:1} and \ref{remark:bounded-regret:2}.
The assumption is nevertheless stated here to match the results provided in the literature for the linear bandit problem, which we discuss in Section~\ref{sec:linear-bandit}.

\begin{lemma}[A bound on the cumulative squared regret] \label{lem:cum-squared-regret}
Suppose Assumptions~\ref{ass:lipschitz-l} and \ref{ass:bounded-regret} hold.
Furthermore, assume $\thetatrue \in \thetaSet$ and $\Vert \thetatrue- \mu_{n}\Vert_{\Lambda_{n}} \leq \gamma_{n}(\delta)$ for all $n \in \NsmallerN$.
Then, the cumulative squared regret is bounded as follows:
\begin{align}\label{eq:lem:cum-squared-regret}
\sum_{n=0}^{N-1} r_n^2
&\leq
2 \max(c_r, 2 L_z \gamma_N(\delta))^2  \log(\det(\Lambda_0\inv \Lambda_N))
\end{align}
with constants $L_z$ and $c_r$ as defined in Assumptions \ref{ass:lipschitz-l} and \ref{ass:bounded-regret}, respectively.
\end{lemma}
\begin{proof}
As $\thetatrue \in \thetaSet$ and $\Vert \thetatrue- \mu_{n}\Vert_{\Lambda_{n}} \leq \gamma_{n}(\delta)$ for all $n \in \NsmallerN$ by assumption, $(\uopttrue_n, \thetatrue)$ is a feasible point of the optimization problem $\problem_n(\delta)$.
With $(u_n, \tilde{\theta}_n)$ a solution to $\problem_n(\delta)$, we directly obtain the following bound on the instantaneous regret:
\begin{align}
0 \leq r_n \leq \objective_n(u_n, \thetatrue) -  \objective_n(u_n, \tilde\theta_n).
\end{align}
Together with Lipschitz continuity of the loss function in $z$, Assumption~\ref{ass:lipschitz-l}, we have
\begin{align}
r_n &\leq l_n(u_n, A_n(u_n)\thetatrue) - l_n(u_n, A_n(u_n)\tilde{\theta}_n) \\
&\leq  L_z \Vert A_n(u_n) \thetatrue - A_n(u_n) \tilde{\theta}_n \Vert_V \label{subeq:regret-bound} \\
&\leq  L_z \left\Vert \tilde{A}_n \right\Vert_2 \left\Vert \thetatrue - \tilde{\theta}_n \right\Vert_{\Lambda_{n}},
\end{align}
where we introduced $\tilde{A}_n := V^{\frac{1}{2}} A_n(u_n) \Lambda_{n}^{-\frac{1}{2}}$.
By assumption, $\Vert \thetatrue- \mu_{n}\Vert_{\Lambda_{n}} \leq \gamma_{n}(\delta)$ and $\Vert \tilde\theta_n- \mu_{n}\Vert_{\Lambda_{n}} \leq \gamma_{n}(\delta)$ such that
\begin{align}
\Vert \thetatrue - \tilde{\theta}_n \Vert_{\Lambda_{n}} \leq \Vert \thetatrue - \mu_{n} \Vert_{\Lambda_{n}} + \Vert \tilde{\theta}_n - \mu_{n}\Vert_{\Lambda_{n}} \leq 2 \gamma_{n}(\delta).
\end{align}
Note that the sequence $\gamma_n(\delta)$, $n \in \NsmallerN$, is nondecreasing, since $\det(A+B) \geq \det(A)$ for positive semidefinite matrices $A$ and $B$.
Together with the assumption that $r_n \leq c_r$, we directly obtain
\begin{align}
r_n &\leq \min(c_r, 2 L_z \gamma_N(\delta) \Vert \tilde{A}_n \Vert_2) \leq  \max(c_r, 2 L_z \gamma_N(\delta)) \min(1, \Vert \tilde{A}_n \Vert_2).
\end{align}
Now, considering the cumulative squared regret and using the inequality $\log(1 + x) \geq \tfrac{x}{2}$ for $x \in [0,1]$, we obtain
\begin{align} \label{eq:intermediate-3}
\sum_{n=0}^{N-1} r_n^2 \leq
\max(c_r, 2 L_z \gamma_N(\delta))^2 \sum_{n=0}^{N-1} 2\log(1 + \Vert \tilde{A}_n \Vert_2^2).
\end{align}
Note that $\Vert \tilde{A}_n \Vert_2^2 = \Vert \tilde{A}_n \tilde{A}_n\T \, \Vert_2$ and let $M_n = V^{\frac{1}{2}} A_n(u_n) \Lambda_0^{-\frac{1}{2}}$ and $S_n = \eye_{\ntheta} + \sum_{i=0}^{n-1} M_i\T M_i$.
With these definitions, we have $\tilde A_n \tilde{A}_n\T = M_n S_{n}\inv M_n\T$, which allows us to apply Lemma~\ref{lem:potential} to \eqref{eq:intermediate-3} yielding
\begin{align}
\sum_{n=0}^{N-1} r_n^2
&\leq
2 \max(c_r, 2 L_z \gamma_N(\delta))^2  \left(\log(\det(S_N) - \log(\det(S_0)))\right).
\end{align}
Noting that $\det(S_N) = \det(\Lambda_0\inv \Lambda_N)$ and $\log(\det(S_0)) = \log(\det(\eye_{\ntheta})) = 0$, we obtain the desired bound.
\end{proof}

\begin{remark}[Assumption~\ref{ass:bounded-regret} is not required]\label{remark:bounded-regret:1}
	The inequality derived in~\eqref{subeq:regret-bound} under the assumption that $\Vert \thetatrue- \mu_{n}\Vert_{\Lambda_{n}} \leq \gamma_{n}(\delta)$ combined with Assumption \ref{ass:prior} implies
	\begin{align}
		r_n &\le L_z \Vert A_n(u_n) \thetatrue - A_n(u_n) \tilde{\theta}_n \Vert_V \\
		&\le L_z  \Vert V^{\frac{1}{2}} A_n(u_n)\Lambda_0^{-\frac{1}{2}} \Vert_2 \Vert \thetatrue - \tilde{\theta}_n \Vert_{\Lambda_0} \\
		&\le 2L_z  c_\theta \Vert V^{\frac{1}{2}} A_n(u_n)\Lambda_0^{-\frac{1}{2}} \Vert_2.
	\end{align}
	Therefore, Assumption~\ref{ass:bounded-regret} is not needed for this analysis and the constant $c_r$ can simply be replaced with $2L_z c_\theta \Vert V^{\frac{1}{2}} A_n(u_n)\Lambda_0^{-\frac{1}{2}} \Vert_2 $ in inequality \eqref{eq:lem:cum-squared-regret}.
	However, a problem-specific bound $c_r$ derived based on knowledge of $l_n(u, z)$, $\U_n$ and $\thetaSet$ as in Assumption~\ref{ass:bounded-regret} might lead to a different, potentially tighter constant.
	Besides, Assumption~\ref{ass:bounded-regret} is commonly used in the linear bandit literature which allows for a straightforward comparison with their results.
\end{remark}

\begin{theorem}[Data-dependent regret bound] \label{thm:regret-data-dependent}
Suppose Assumptions~\ref{ass:subgaussian}, \ref{ass:prior}, \ref{ass:lipschitz-l} and \ref{ass:bounded-regret} hold and consider the confidence level $\delta \in (0, 1)$.
If $u_n$ is chosen as a solution to $\problem_n(\delta)$
for all $n \in \NsmallerN$, then the following statement holds with a probability of at least $1-\delta$:
\begin{align} \label{eq:R-bound-data-dependent}
R_N
&\leq
\max(c_r, 2 L_z \gamma_N(\delta))  \sqrt{2 N \log\left(\det\left(\Lambda_0\inv \Lambda_{N}\right)\right)},
\end{align}
where $L_z$ and $c_r$ are data-independent constants as defined in Assumption~\ref{ass:lipschitz-l} and \ref{ass:bounded-regret}, respectively.
\end{theorem}

\begin{proof}
We assume $\thetatrue \in \thetaSet$ and $\Vert \thetatrue- \mu_{n}\Vert_{\Lambda_{n}} \leq \gamma_{n}(\delta)$ for all $n \in \NsmallerN$.
Under this assumption, the Cauchy-Schwarz inequality together with Lemma~\ref{lem:cum-squared-regret} implies
\begin{align}
R_N = \sum_{n=0}^{N-1} r_n \leq \sqrt{N \sum_{n=0}^{N-1} r_n^2} \leq \max(c_r, 2 L_z \gamma_N(\delta)) \sqrt{2 N \log\left(\det\left(\Lambda_0\inv \Lambda_N\right)\right)}.
\end{align}

Theorem~\ref{thm:sim-conf-bound-data-dependent-constrained} implies that the assumption $\thetatrue \in \thetaSet$ and $\Vert \thetatrue- \mu_{n}\Vert_{\Lambda_{n}} \leq \gamma_{n}(\delta)$ for all $n \in \NsmallerN$ holds with a probability of at least $1-\delta$.
As a consequence, the above bound holds with the same probability, which concludes the proof.
\end{proof}

Note that the regret bound in Theorem~\ref{thm:regret-data-dependent} relies on the data-dependent values $\Lambda_N$ and $\gamma_N(\delta)$.
Bounding these two terms allows us to derive a data-independent bound.
To this end, we make the following additional assumption.

\begin{assumption}[Bounded regressors] \label{ass:bounded-regressors}
We assume that the weighted regressors are bounded, i.e., there is a constant $c_A$ such that
$\Vert V^{\frac{1}{2}} A_n(u) \Lambda_0^{-\frac{1}{2}} \Vert_{\mathrm{F}} \leq c_A$,
for all $u \in \U_n$, $n \in \NsmallerN$.
\end{assumption}

\begin{theorem}[Data-independent regret bound] \label{thm:regret-data-independent}
Suppose Assumptions~\ref{ass:subgaussian}, \ref{ass:prior}, \ref{ass:lipschitz-l}, \ref{ass:bounded-regret}, \ref{ass:bounded-regressors} hold and consider the confidence level $\delta \in (0, 1)$.

If $u_n$ is chosen as a solution to $\problem_n(\delta)$ for all $n \in \NsmallerN$, then the following statement holds with a probability of at least $1-\delta$:
\begin{align}\label{eq:thm:regret-data-independent}
R_N
&\leq
\max(c_r, 2 L_z \bar{\gamma}_N(\delta)) \sqrt{2 N \ntheta \log\left(1 + \frac{N c_A^2}{\ntheta}\right)},
\end{align}
where
\begin{align} \label{eq:gamma-independent-gamma}
\bar{\gamma}_N(\delta) = \sqrt{c_\theta^2 + c_v^2 \ntheta \log\left(1 + \frac{N c_A^2}{\ntheta}\right) + 2 c_v^2 \log\left(\frac{1}{\delta}\right)},
\end{align}
with $c_v$, $c_\theta$, $L_z$, $c_A$ and $c_r$ data-independent constants as defined in Assumption~\ref{ass:subgaussian}, \ref{ass:prior}, \ref{ass:lipschitz-l}, \ref{ass:bounded-regret} and \ref{ass:bounded-regressors}, respectively.
\end{theorem}

\begin{proof}
The only data-dependent term in \eqref{eq:data-dependent-gamma-constrained} and \eqref{eq:R-bound-data-dependent} in Theorem~\ref{thm:sim-conf-bound-data-dependent-constrained} and \ref{thm:regret-data-dependent}, respectively, is $\log\left(\det\left(\Lambda_0\inv \Lambda_N\right)\right)$.
First, note that the determinant of $\Lambda_0\inv \Lambda_N$ can be expressed as
\begin{align}
\det\left(\Lambda_0\inv \Lambda_N\right) = \det\left(\eye_{\ntheta} + \sum_{n=0}^{N-1} M_n\T M_n\right)
\end{align}
for $M_n = V^{\frac{1}{2}} A_n(u_n) \Lambda_0^{-\frac{1}{2}}$.
Applying Lemma~\ref{lem:logdet} then yields the desired data-independent bound.
\end{proof}

\begin{remark}[A data-independent regret bound without Assumption~\ref{ass:bounded-regret}]\label{remark:bounded-regret:2}
	As discussed in Remark~\ref{remark:bounded-regret:1}, the constant $c_r$ can be replaced with $2L_z  c_\theta \Vert V^{\frac{1}{2}} A_n(u_n)\Lambda_0^{-\frac{1}{2}} \Vert_2 $, which is in turn bounded by $2L_z c_\theta c_A$ under Assumption~\ref{ass:bounded-regressors}.
	Therefore, the data-independent regret bound \eqref{eq:thm:regret-data-independent} can be replaced by
	\begin{align}\label{eq:thm:regret-data-independent:bis}
		R_N \leq
		2 L_z  \max(c_\theta c_A, \bar{\gamma}_N(\delta)) \sqrt{2 N \ntheta \log\left(1 + \frac{N c_A^2}{\ntheta}\right)}.
	\end{align}
	Note that this bound does not require Assumption~\ref{ass:bounded-regret} anymore.
\end{remark}

The above theorem shows that the cumulative regret behaves as $\bigO\left(\sqrt{N} \log(N)\right)$ asymptotically.
In the special case of finite action sets, we can even achieve a regret that behaves as $\bigO\left(\log(N)^2\right)$.

\begin{theorem}[Regret bounds for finite action sets] \label{thm:regret-finite-action-set}
Suppose the action sets $\U_n$ are finite for all $n \in \NsmallerN$.
We define the gap $\Delta_n > 0$ as
\begin{align} \label{eq:gap}
\Delta_n = \min_{u \in \check{\U}_n}\objective_n(u, \thetatrue) - \objective_n(\uopttrue_n, \thetatrue) \quad \text{where} \quad \check{\U}_n = \left\lbrace u \in \U_n \mid \objective_n(u, \thetatrue) > \objective_n(\uopttrue_n, \thetatrue)\right\rbrace,
\end{align}
where we define $\Delta_n = +\infty$ whenever $\check{\U}_n$ is empty.
Furthermore, let $\bar \Delta_N = \min_{n \in \NsmallerN} \Delta_n$.

Suppose Assumptions~\ref{ass:subgaussian}, \ref{ass:prior}, \ref{ass:lipschitz-l} and \ref{ass:bounded-regret} hold and consider the confidence level $\delta \in (0, 1)$.
If $u_n$ is chosen as a solution to $\problem_n(\delta)$
for all $n \in \NsmallerN$, then the following data-dependent bound holds with a probability of at least $1-\delta$:
\begin{align}
R_N
&\leq
\frac{2 \max(c_r, 2 L_z \gamma_N(\delta))^2}{\bar \Delta_N}  \log(\det(\Lambda_0\inv \Lambda_N)),
\end{align}
where $L_z$, $c_r$ are data-independent constants as defined in Assumption~\ref{ass:lipschitz-l} and \ref{ass:bounded-regret}, respectively.

If, in addition, Assumption~\ref{ass:bounded-regressors} is satisfied, then the following data-independent bound holds with a probability of at least $1-\delta$:
\begin{align}
R_N
&\leq
\frac{2 \ntheta \max(c_r, 2 L_z \bar \gamma_N(\delta))^2}{\bar \Delta_N}  \log\left(1 + \frac{N c_A^2}{\ntheta}\right),
\end{align}
with $\bar \gamma_N(\delta)$ as in \eqref{eq:gamma-independent-gamma} and data-independent constant $c_A$ as defined in Assumption~\ref{ass:bounded-regressors}.
\end{theorem}

\begin{proof}
First, note that we have either $r_n = 0$ or $r_n \geq \bar \Delta_N$ for all $n \in \NsmallerN$.
The cumulative regret is thus bounded by
\begin{align}
R_N = \sum_{n=0}^{N-1} r_n \leq \sum_{n=0}^{N-1} \frac{r_n^2}{\bar \Delta_N}.
\end{align}
Applying Lemma~\ref{lem:cum-squared-regret} yields the desired data-dependent bound
\begin{align}
R_N \leq \frac{2 \max(c_r, 2 L_z \gamma_N(\delta))^2}{\bar \Delta_N}   \log(\det(\Lambda_0\inv \Lambda_N)).
\end{align}
As before, applying Lemma~\ref{lem:logdet} yields the corresponding data-independent bound.
\end{proof}

\begin{remark}[Polytopic action sets and concave objectives]
	The result from the previous theorem also applies in the case where the action set $\U_n$ is a polytope and the objective $\objective_n(u, \theta)$ is concave in $u$ for all $n \in \NsmallerN$ and all $\theta \in \thetaSet$.
	More precisely, the regret bound in Theorem~\ref{thm:regret-finite-action-set} holds with gaps $\Delta_n$ defined on the vertices of the polytopic set $\U_n$.
	The reason is that, as $u_n$ is the solution to a concave problem, one can choose it to be an extreme point of $\U_n$.
	When $\U_n$ is a polytope, the extreme points coincide with the vertices of the polytope, which are finite.
\end{remark}

From the high-probability bounds discussed so far, we can directly obtain a bound on the expected cumulative regret, as shown by the following corollary.

\begin{corollary}[A bound on the expected cumulative regret]
Suppose the assumptions of Theorem~\ref{thm:regret-data-dependent} or \ref{thm:regret-data-independent} hold and let $b_N(\delta)$ be the corresponding high-probability bound.
Then, the expected cumulative regret is bounded as follows:
\begin{align}
\expected{}{R_N} \leq b_N(\delta) + \delta N c_r.
\end{align}
If $N$ is known a priori, choosing $\delta = \tfrac{1}{N}$ yields an expected regret $\expected{}{R_N}$ of the order $\bigO\left(\sqrt{N} \log(N)\right)$.

Similarly, under the assumptions of Theorem~\ref{thm:regret-finite-action-set}, choosing $\delta  = \tfrac{1}{N}$ yields an expected regret $\expected{}{R_N}$ of the order $\bigO\left(\log(N)^2\right)$.
\end{corollary}
\begin{proof}
From Theorem~\ref{thm:regret-data-dependent} or \ref{thm:regret-data-independent}, we have $R_N \leq b_N(\delta)$ with probability at least $1-\delta$.
If $R_N > b_N(\delta)$, which occurs with probability at most $\delta$, we bound $R_N$ by the maximum cumulative regret $N c_r$.
Using conditional expectation, this directly implies
\begin{align}
\expected{}{R_N} \leq (1-\delta )b_N(\delta) + \delta N c_r
 \le b_N(\delta) + \delta N c_r.
\end{align}
\end{proof}

\section{The linear bandit problem}
\label{sec:linear-bandit}

In this section, we compare the regret bound derived in the previous section with the state of the art for the linear bandit problem, which is obtained as a special case of the gray-box optimization setup.

If the output is scalar, the loss function is equal to the output $z$ and if the knowledge about $\thetaSet$ is used neither within the estimation problem nor within the optimistic problem defining the next action, we recover the classic LCB method for the linear bandit problem \citep{Dani2008,Abbasi-Yadkori2011} or, more generally, the linear contextual bandit problem \citep{Li2010}:

\begin{example}[Linear bandits]
We obtain the linear contextual bandit problem by considering a single-output model $y = a(u)\T \theta + v$ and the time-invariant linear loss function $l(u, z) = z$.
More specifically, we obtain the linear bandit problem by additionally assuming a time-invariant model of the form $y = u\T \theta + v$.
\end{example}

We first summarize the assumptions typically made for the analysis of the linear bandit problem \citep{Abbasi-Yadkori2011,Lattimore2020}.
\begin{assumption} \label{ass:linear-bandit}
Consider the linear bandit problem, i.e. $y = u\T \theta + v$ and $l(u, z) = z$.
Suppose Assumption~\ref{ass:subgaussian} holds.
In addition, we make the following assumptions:
\begin{enumerate}
	\item[(i)] $V = 1$, $\Lambda_0 = \lambda_0 \eye_{\ntheta}$, $\lambda_0 > 0$, and $\mu_0 = 0$. %
	\item[(ii)] $\Vert \thetatrue \Vert_2 \leq c_S$. %
	\item[(iii)] $\Vert u \Vert_2 \leq c_L$ for all $u \in \U_n$, $n \in \NsmallerN$. %
	\item[(iv)] $u\T \thetatrue \in [-1, 1]$ for all $u \in \U_n$, $n \in \NsmallerN$.
\end{enumerate}
\end{assumption}

Note that (iv) directly implies that the regret is bounded by 2, as $(u - \uopttrue)\T \thetatrue \leq 2 = c_r$.

Theorems~\ref{thm:regret-data-dependent} and \ref{thm:regret-data-independent} imply the following bounds for the linear bandit problem under Assumption~\ref{ass:linear-bandit}.

\begin{corollary}[Regret bounds for linear bandits]
Suppose Assumption \ref{ass:linear-bandit} holds and consider the confidence level $\delta \in (0, 1)$.
If $u_n$ is chosen as a solution to $\problem_n(\delta)$ for all $n \in \NsmallerN$ with $\gamma_n(\delta)$ as in \eqref{eq:data-dependent-gamma-constrained}, then the following statements hold with a probability of at least $1-\delta$:
\begin{align}
\intertext{Data-dependent bound:}
R_N
&\leq
2 \max(1, \gamma_{N}(\delta)) \sqrt{2 N \log\left(\det\left(\Lambda_0\inv \Lambda_N\right)\right)}
\intertext{Data-independent bound:}
R_N
&\leq
2 \max(1, \bar\gamma_{N}(\delta))\sqrt{2 N d \log\left(1 + \frac{N c_L^2}{d \lambda_0}\right)}
\end{align}
where $\bar\gamma_{N}(\delta) = \sqrt{{c_S^2}{\lambda_0} + c_v^2 d \log\left(1 + \frac{N c_L^2}{d \lambda_0}\right) + 2 c_v^2 \log\left(\tfrac{1}{\delta}\right)}$.
\end{corollary}

We point out that the data-dependent bound matches the bound reported in \cite{Abbasi-Yadkori2011} and \cite{Lattimore2020}.
They, however, use
\begin{align}
\gamma^\prime_N(\delta) = \sqrt{\lambda_0} c_S + c_v \sqrt{\log\left(\det\left(\Lambda_0\inv \Lambda_{N}\right)\right) + 2 \log\left(\frac{1}{\delta}\right)}
\end{align}
in their definition of the parameter confidence sets.
Analogously, the data-independent bound given in \cite{Lattimore2020} uses
\begin{align}
\bar\gamma^\prime(\delta)  = \sqrt{\lambda_0} c_S + c_v \sqrt{d \log\left(1 + N \frac{c_L^2}{\lambda_0 d}\right) + 2 \log\left(\frac{1}{\delta}\right)}.
\end{align}
We point out that $\gamma^\prime_N(\delta)$, $\bar\gamma^\prime(\delta)$ are strictly larger than $\gamma_n(\delta)$, $\bar\gamma_{N}(\delta)$.
Our regret analysis thus improves upon the state of the art for the linear stochastic bandit due to improved parameter confidence bounds.

\begin{remark}
We emphasize that the state-of-the-art regret analysis of the LCB approach for linear bandits requires the assumption $\Vert \thetatrue \Vert_2 \leq c_S$, i.e., a priori knowledge of a bounded set of admissible parameters $\thetaSet$ is assumed.
The standard LCB algorithm, however, does not leverage this assumption, while the proposed method does.
We illustrate the potential performance gain resulting from exploitation of the assumption $\theta \in \thetaSet$ for a linear bandit problem in Section~\ref{sec:example-linear-bandit}.
\end{remark}

\begin{remark}[Lower bound on the regret]
For some classes of action sets, \cite{Lattimore2020} have shown that the cumulative regret $R_N$ of any policy is lower bounded by $c \sqrt{N}$ for some constant $c$.
A similar result is given in \cite{Rusmevichientong2010}.
These results highlight the near-optimality of the proposed method.
\end{remark}

\section{Illustrative examples}
\label{sec:experiments}

In the following, we illustrate the behavior of the proposed method through several examples.
We begin by comparing the method with its structure-agnostic variant in two simple cases.
The first example shows how leveraging prior knowledge about the true parameter -- encoded in the set $\thetaSet$ -- can substantially improve performance in a linear bandit setting.
The second example highlights gains over the structure-agnostic approach by exploiting structure induced by a known loss function.
As a third, more applied example, we consider a steel recycling problem, demonstrating the effectiveness of the proposed algorithm in a time-varying setting.

\subsection{Exploiting structure given by the set of admissible parameters}
\label{sec:example-linear-bandit}

\begin{figure}
\centering
\includegraphics[width=0.9\linewidth]{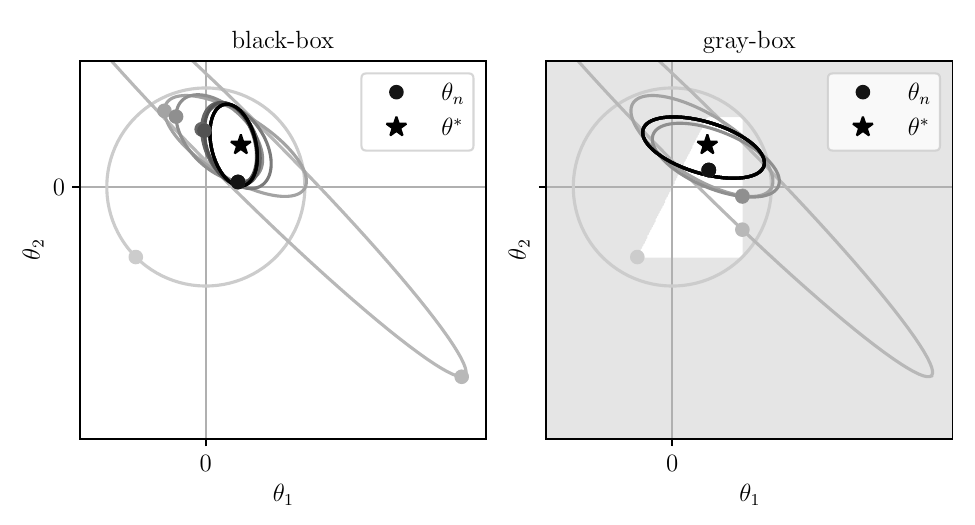}
\caption[short]{Linear bandit with additional structural knowledge of the true parameter:
Confidence ellipsoid as well as the parameter values $\theta_n$ picked by the optimizer at iteration $n$.
The shaded area indicates the infeasible region with respect to the constraint $\theta \in \thetaSet$.
}
\label{fig:bandit-constrained}
\end{figure}

We consider a linear bandit problem in two dimensions where we assume additional structural knowledge of the true parameter is given by the set of admissible parameters $\thetaSet$.
We compare the standard structure-agnostic LCB approach, which does not leverage the assumption $\theta \in \thetaSet$, to the proposed method.

Consider $f(u, \theta) = u\T \theta$, $\U = [0, 1]^2$ and $l(u, z) = z$.
The measurement noise is normally distributed with variance $\sigma_v^2 = 0.04$.
The estimator uses $\mu_0 = (0, 0)$, $\Lambda_0 = 0.5\eye_2$ and $V = \tfrac{1}{\sigma_v^2}$.
We assume the following set of admissible parameters is given
\begin{align}
\thetaSet = \left\lbrace \theta \in [-1, 1]^2 \,|\, \theta_2 - 2\theta_1 \leq 0 \right\rbrace.
\end{align}

We use the data-dependent values $\gamma_n(\delta)$ as defined in \eqref{eq:data-dependent-gamma-constrained} with $c_\theta = 1$, $c_v = 1$, $\delta = 0.05$ and set $\gamma_0 = c_\theta$.
Figure~\ref{fig:bandit-constrained} illustrates the confidence sets as well as the parameter values chosen by the optimizer for each time step.
For the structure-exploiting approach, the gray region indicates values of $\theta$ that are not in the admissible set $\thetaSet$.
Note that it suffices to get the correct sign of the parameter in order to not induce any regret.
The structure-agnostic approach shown on the left picks values in the upper left quadrant, which are prohibited by the additional constraint $\theta \in \thetaSet$ for the structure-exploiting method resulting in a smaller cumulative regret.

The superior performance of the structure-exploiting approach is underpinned by Figure~\ref{fig:bandit-constrained-regret} illustrating the regret distribution over 300 independent simulations where in each simulation a new value of the true parameter is sampled uniformly from $\thetaSet$.
The average cumulative regret is reduced from 1.9 for the black-box approach to 1.2 for the gray-box approach corresponding to a reduction of over 35\%.

\begin{figure}
\centering
\includegraphics[width=0.9\linewidth]{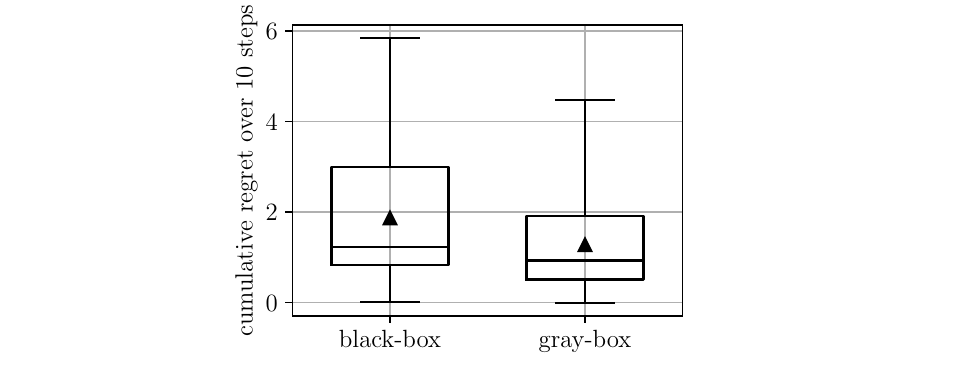}
\caption[short]{Linear bandit with additional structural knowledge of the true parameter:
Comparison of the cumulative regret obtained over 300 independent simulations where the true parameter is uniformly sampled from $\thetaSet$.
The whiskers indicate the minimum and maximum values, and the triangle indicates the mean.}
\label{fig:bandit-constrained-regret}
\end{figure}

\subsection{Exploiting structure given by the loss}

With the second example, we illustrate the performance gain obtained from exploiting known problem structures given by the loss function.
To this end, consider the following problem setup:

\begin{example} \label{ex:minimal}
Regard the parametric model $z = A(u)\theta \in \R^2$, where $\theta \in \R^4$ and
\begin{align} \label{eq:model-example}
A(u) = \begin{bmatrix}
 u & 1 & 0 & 0 \\
 0 & 0 & u & 1 \\
\end{bmatrix}\!.
\end{align}
The true parameter is $\thetatrue = (-.7, 0.3, -0.3, 0.4)$.
The loss function is $l(u, z) = z_1^2 + 0.1z_2^2$.
The scalar action $u$ is constrained to be in the set $\U = [-1, 1]$.
\end{example}

\begin{figure}
	\centering
	\includegraphics[width=0.9\linewidth]{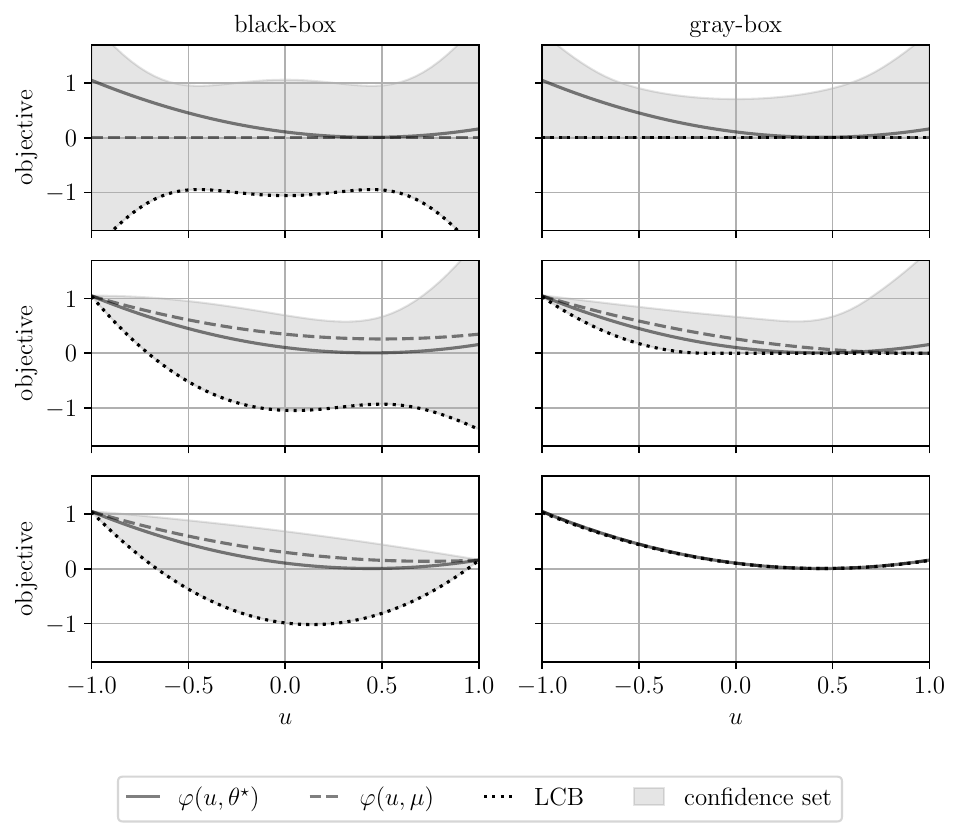}
	\caption{Comparison of the structure-agnostic and the structure-exploiting approach for Example~\ref{ex:minimal}.
	The top row shows the prior confidence set for the objective, while the middle row shows the confidence set after evaluating at $u_1 = -1$ and observing the model output $y_1 = (1., 0.7)$ for the proposed method and the objective value $y_1^{\mathrm{BB}} = 1.05$ for the structure-agnostic approach.
	The bottom row shows the confidence set after the additional evaluation at $u_2 = 1$ and observing $y_2 = (-0.4, 0.1)$ and $y_2^{\mathrm{BB}} = 0.16$, respectively.}
	\label{fig:samples-and-lcb}
\end{figure}

For the structure-exploiting gray-box approach, we use the parametric model $f(u, \theta)$ given in \eqref{eq:model-example} and assume the loss function $l$ is known.
In addition, we consider the black-box approach arising from $f_\mathrm{BB}(u, \theta_{\mathrm{BB}}) = b(u)\T \theta_{\mathrm{BB}}$ and $l_{\mathrm{BB}}(z) = z$ with
$b(u) = (u^2, u, 1) \in \R^3$ and $\theta_{\mathrm{BB}} \in \R^3$.
Note that there is $\theta^\star_{\mathrm{BB}}$ such that $f_\mathrm{BB}(u, \theta^\star_{\mathrm{BB}}) = \objective(u, \thetatrue)$, i.e., both approaches can recover the true objective.
In order to compare the structure-exploiting method to the structure-agnostic approach, we consider a simulation without measurement noise, i.e.,
we directly observe $y = f(u, \thetatrue)$ and $y^{\mathrm{BB}} = f_{\mathrm{BB}}(u, \thetatrue_\mathrm{BB}) = \objective(u, \thetatrue)$, respectively.

We use $\mu_0 = \mathbf{0}_4$, $\Lambda_0 = 10 \eye_4$, $V = 10^6 \eye_2$  and $\mu_0^{\mathrm{BB}} = \mathbf{0}_3$, $\Lambda_0^{\mathrm{BB}} = 10\eye_3$, $V^{\mathrm{BB}} = 10^6$, and $c_\theta = 1$, $c_v = 1$, $\delta = 0.05$.
For both approaches, we consider parameters that lead to an objective function with its integral bounded as follows:
\begin{align}
\thetaSet = \left\lbrace \theta \in \R^4:  \int_{-1}^1 l(f(u, \theta)) \mathrm{d}u \leq 1 \right\rbrace,~~
\thetaSet_{\mathrm{BB}} = \left\lbrace \theta_{\mathrm{BB}} \in \R^3:  \int_{-1}^1 l_{\mathrm{BB}}(f_{\mathrm{BB}}(u, \theta_\mathrm{BB})) \mathrm{d}u \leq 1 \right\rbrace.
\end{align}

On the one hand, $\mathrm{dim}(\theta) > \mathrm{dim}(\theta_{\mathrm{BB}})$, i.e., the gray-box approach needs to identify more unknown parameters.
On the other hand, the gray-box approach observes a two-dimensional measurement at each time step, while in the black-box setup only a scalar observation of the objective function is available.

Figure~\ref{fig:samples-and-lcb} illustrates the confidence sets for the unknown objective for both the gray-box and the black-box approach.
The confidence sets for the objective are bounded below by zero in the gray-box setting, highlighting the fact that the structure induced by $l(z)$ is preserved.

As the first action, we pick $u_1 = -1$, which is a minimizer of the acquisition functions for both methods.
The middle row shows the updated confidence sets and acquisition functions after evaluation at $u_1 =-1$ and observing the model output $y_1 = (1., 0.7)$ for the proposed method and the objective value $y_1^{\mathrm{BB}} = 1.05$ for the black-box method.
We pick $u_2 = 1$ as the next evaluation point, which is a minimizer of the acquisition function in both approaches.
After receiving the second output measurement $y_2 = (-0.4, 0.1)$, the gray-box approach has identified the correct model with the confidence set for the output, including only the true objective.
Thus, the next iteration would deliver the correct solution $\uopttrue$.
The black-box approach would need one more iteration to identify the correct parameters and an additional one to find $\uopttrue$.

\subsection{Application example}
\label{sec:example-steel}

\begin{figure}
	\centering
	\includegraphics[width=0.9\linewidth]{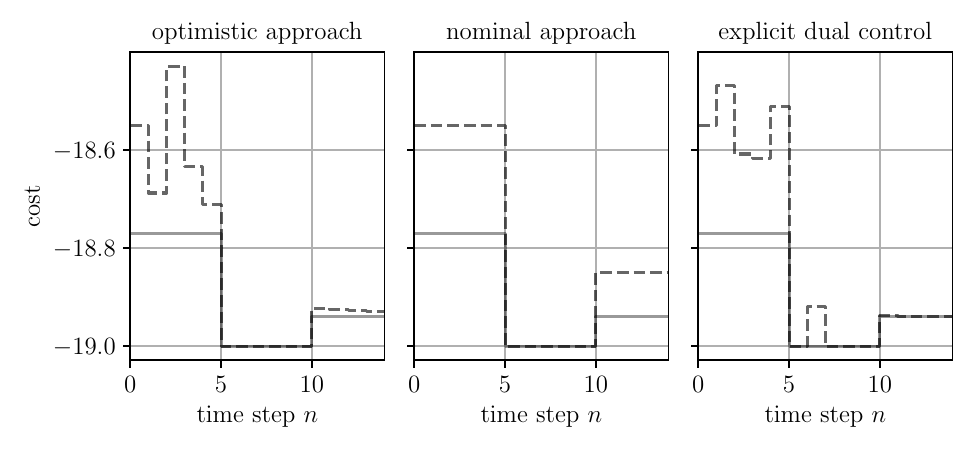}
	\caption{Cost incurred by the proposed optimistic method in comparison to the nominal and explicit dual approach for the steel recycling example.}
	\label{fig:cost-steel}
\end{figure}

\begin{figure}
	\centering
	\includegraphics[width=0.9\linewidth]{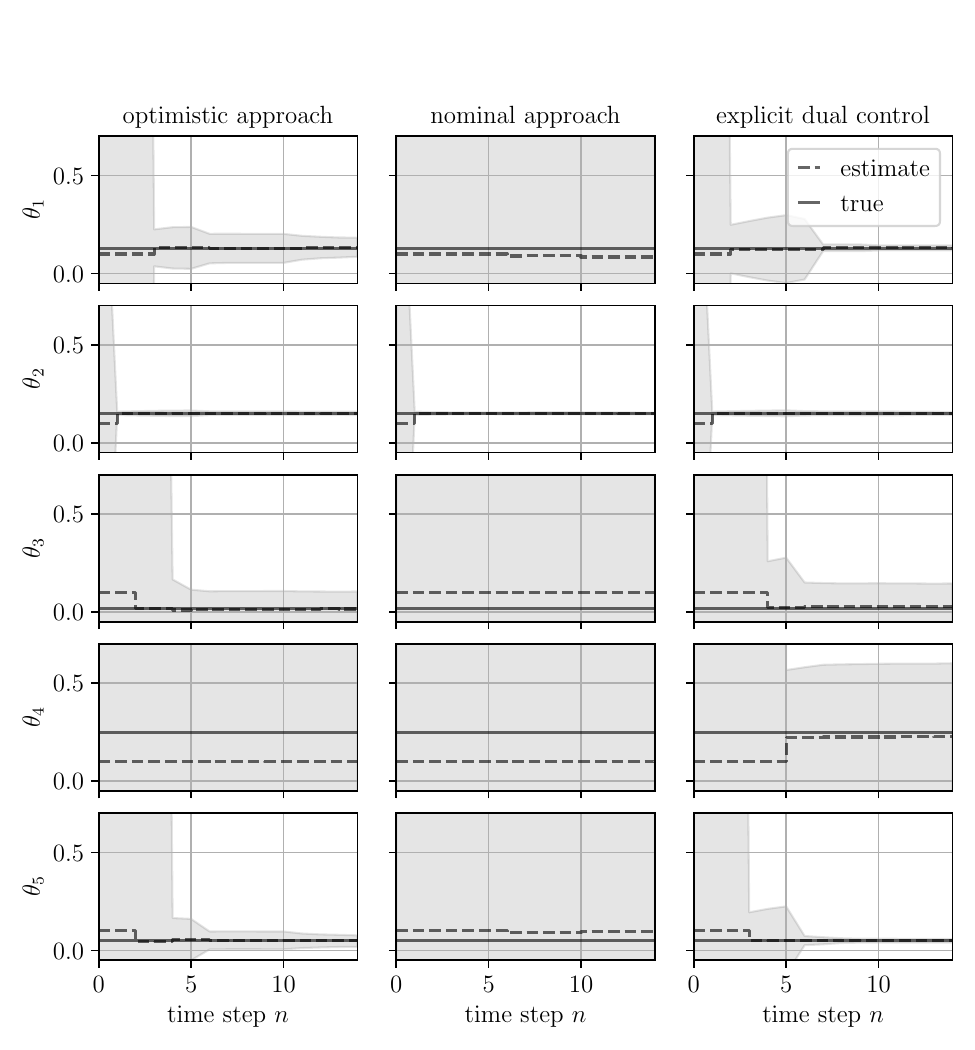}
	\caption{Concentration estimates together with the confidence bounds for the steel recycling example.}
	\label{fig:estimation-steel}
\end{figure}

As a final example, we apply the proposed method to a variant of the steel recycling problem introduced in \cite{Ghezzi2023}:
We consider a scrapyard where the scrap is divided into heaps.
The aim is to produce steel with a low concentration of the pollutant copper while minimizing the price of the scrap mix.
The pollutant concentration in each heap is not known a priori but needs to be estimated based on measurements of the pollutant concentration in the molten scrap mix.
The steel recycling problem is formalized as follows:
Let $\thetatrue \in \thetaSet$ denote the true concentration of the pollutant in each scrap heap, i.e., $\ntheta = \dim(\theta)$ corresponds to the number of heaps.
At each iteration $n$ and given the prices $c_n \in \R_+^{\ntheta}$ of the material in each heap, we choose actions $u \in \U = \left\lbrace u \in [0, 1]^{\ntheta} \, | \, \mathbf{1}\T u = 1 \right\rbrace$ corresponding to the amount of scrap picked from each heap.
The output is given by $f(u, \thetatrue) = u\T \thetatrue$ and corresponds to the concentration of pollutant in the scrap mix.
The cost function is
\begin{align}
l_n(u, z) = c_n\T \, u - r(z) ~ \text{ with } r(z) = r_{\max} - r_{\mathrm{slope}}\max(0, z - z_{\mathrm{target}}),
\end{align}
where the first term represents the cost of the selected scrap mix, while the second term corresponds to the revenue generated from selling the steel.
The revenue is equal to $r_{\max}$ when the pollutant concentration of the steel is below the target level $z_{\mathrm{target}} = 0.12$.
If the pollutant concentration exceeds this threshold, the revenue decreases linearly with a slope of $r_{\mathrm{slope}} = 15$.
We implement the nondifferentiable cost with a reformulation using an additional slack variable.

We use the admissible parameter set $\thetaSet = [0, 1]^{\ntheta}$ which encodes the fact that the true parameter vector corresponds to a vector of concentrations which need to take values in $[0, 1]$.
The estimator uses $\mu_0 = 0.1 \cdot \mathbf{1}_{\ntheta}$, $\Lambda_0 = \eye_{\ntheta}$ and $V = \frac{1}{\sigma_v^2}$ where $\sigma_v^2 = 10^{-6}$ is the variance of the Gaussian measurement noise such that we can use $c_v = 1$.

We compare the proposed optimistic gray-box approach with a nominal certainty-equivalent approach and an explicit dual approach.
The nominal approach selects $u_n$ as a solution to
\begin{mini}|s|
	{\scriptstyle{u \in \U}}
	{l_n\left(u, f(u, \mu_{n})\right).}
	{}
	{}
\end{mini}
The explicit dual control approach selects $u_n$ as a solution to
\begin{mini}|s|
	{\scriptstyle{u \in \U}}
	{l_n\left(u, f(u, \mu_{n})\right) + \beta \trace\left(\Lambda_{\mathrm{next}}(\Lambda_{n}, u_n)\inv\right),}
	{}
	{}
\end{mini}
with $\Lambda_{\mathrm{next}}(\Lambda, u) = \Lambda + A(u)\T V A(u)$ and weighting parameter $\beta$.
The additional trace term explicitly encourages exploration.

We consider a simulation over 15 time steps where the prices $c_n$ change every 5 time steps, rendering the problem time-variant.
The true pollutant concentrations are $\thetatrue = (0.13, 0.15, 0.02, 0.25, 0.05)$.
For the first five time steps, the prices are $c_n = (2, 1, 2, 3.5, 2)$.
Then, the prices are given by $c_n = (1, 1.5, 2, 3.5, 1)$ for the next five time steps.
For the final five time steps, the prices are $c_n = (2.2, 0.7, 3.2, 3.5, 1.9)$.
Importantly, note that material from the fourth heap is the most expensive throughout the simulation.

The performance of the three approaches is shown in Figure~\ref{fig:cost-steel}.
The optimistic approach yields a cumulative regret of 0.90,
while the nominal and explicit dual control approaches result in a cumulative regret of 1.55 and 1.18, respectively.
In addition, Figure~\ref{fig:estimation-steel} illustrates the estimated concentrations and the corresponding confidence bounds.
The nominal approach only exploits and never explores resulting in large confidence sets and overall suboptimal performance.
Comparing the proposed optimistic approach with the explicit dual control approach, we point out that the optimistic approach never picks scrap from the fourth -- very expensive -- heap yielding high uncertainty for $\theta_4$, while the explicit dual control approach reduces uncertainty for all the parameters.
This difference in behavior demonstrates that the optimistic approach implicitly restricts exploration to directions that promise improvement with respect to the objective.
In contrast, the explicit dual control approach yields undirected exploration, aiming to reduce uncertainty regardless of its impact on the objective.

\section{Conclusions and Outlook}
\label{sec:conclusion}

This paper considers sequential gray-box optimization where, at each time step, an objective function given in terms of the composition of a loss function and a parametric model is to be minimized.
While both the loss function and the parametric model are known, the parameters are unknown and need to be estimated from noisy measurements of the multivariate model outputs.
The gray-box optimization problem generalizes the fully black-box setting and thus covers the linear bandit and linear contextual bandit problems as special cases.

The proposed algorithm leverages the principle of \textit{optimism in the face of uncertainty} to trade off exploration and exploitation in order to ensure informative measurements are obtained, while keeping the incurred loss small.
In addition to the structure provided by the known loss function, the proposed method leverages additional prior knowledge on the parameter given in terms of a set of admissible parameters $\thetaSet$.
An in-depth regret analysis is provided, which extends established results from linear bandit theory and, in particular, improves upon the state-of-the-art regret bound for the (contextual) linear bandit due to the use of an improved data-dependent finite sample bound for multi-output linear least-squares estimation.
Numerical examples illustrate the effectiveness of the structure-exploiting approach compared to structure-agnostic methods.

Future work will address an extension to sequential gray-box optimization problems with joint input-output constraints, where additional mechanisms are required to ensure safe exploration.

\acks{%
This research was supported by DFG via projects 504452366 (SPP 2364), 560056112 (robust MPC), 535860958 (ALeSCo) and 525018088 (MAWERO), and by BMWK via 03EN3054B.
}

\appendix

\section{Additional Lemmas}

In the following, we prove two additional lemmas required for the proofs of Theorems~\ref{thm:regret-data-dependent} and \ref{thm:regret-data-independent}.

\begin{lemma} \label{lem:potential}
Consider $M_0, \ldots, M_{N-1} \in \R^{k \times d}$ and let $S_n = \eye_d + \sum_{i=0}^{n-1} M_i\T M_i$ for $n \in \NsmallerN$.
We then have
\begin{align}
\log\left(1 + \rho\left(M_n S_{n}\inv M_n\T\right)\right) \leq \log(\det(S_{n+1})) - \log(\det(S_{n})).
\end{align}
\end{lemma}

\begin{proof}
First note that for any positive semidefinite matrix $B \in \R^{k\times k}$ with eigenvalues $\lambda_1, \ldots, \lambda_k \geq 0$, we have
\begin{align}
\log\left(1 + \rho(B)\right) = \log\left(1 + \max_j \lambda_j\right) \leq \sum_{j=1}
^k \log\left(1+\lambda_j\right) = \log\left(\det\left(\eye_k + B\right)\right).
\end{align}
We therefore have
\begin{align}
\log\left(1 + \rho\left(M_n S_{n}\inv M_n\T\right)\right)
& \leq \log(\det(\eye_k + M_n S_{n}\inv M_n\T)) \\
& = \log\left(\det\left(\eye_d + M_n\T M_n S_{n}\inv \right)\right) \\
& = \log\left(\det\left(S_{n} + M_n\T M_n\right)\det\left( S_{n}\inv \right)\right) \\
& = \log\left(\det\left(S_{n+1}\right)\right) - \log\left(\det\left( S_{n} \right)\right),
\end{align}
where the first equality follows from the Weinstein-Aronszajn identity providing $\det(\eye_k + AB) = \det(\eye_d + BA)$, $A \in \R^{k \times d}$, $B \in \R^{d \times k}$.
\end{proof}

\begin{lemma} \label{lem:logdet}
Let $M_0, \ldots, M_{N-1} \in \R^{k \times d}$ be matrices such that $\Vert M_n \Vert_{\mathrm{F}} \leq c$ for all $n \in \NsmallerN$.
Define $S_n = \eye_d + \sum_{i=0}^{n-1} M_i\T M_i$.
We then have
\begin{align}
\log\left( \det\left( S_n\right)\right) \leq d \log\left(1 + \frac{n c^2}{d} \right).
\end{align}
\end{lemma}
\begin{proof}
First note that for any positive definite matrix $B \in \R^{d \times d}$ with positive eigenvalues $\lambda_1, \ldots, \lambda_d > 0$,
Jensen's inequality implies
\begin{align}
\log\left(\det\left(B\right)\right) = \sum_{j=1}^d \log\left(\lambda_j\right) \leq d \log\left(\tfrac{1}{d} \sum_{j=1}^d \lambda_j\right) = d \log\left(\tfrac{1}{d} \trace\left(B\right)\right).
\end{align}
Applying this result to $B = S_n$, we obtain
\begin{align}
\log\left(\det\left(S_n\right)\right)
& \leq d \log\left(\tfrac{1}{d} \trace\left(\eye_d + \sum_{i=0}^{n-1} M_i\T M_i \right)\right) \\
& =  d \log\left(1 + \tfrac{1}{d} \trace\left(\sum_{i=0}^{n-1} M_i\T M_i \right)\right) \\
& \leq  d \log\left(1 + \frac{n c^2}{d} \right),
\end{align}
where the last inequality follows from the assumption $\Vert M_n \Vert_{\mathrm{F}} \leq c$.
\end{proof}

\bibliography{extracted.bib}

\end{document}